\documentclass[submission,copyright,creativecommons]{eptcs}
\providecommand{\event}{ACT 2024} 

\pdfoutput=1

\usepackage[utf8]{inputenc}
\DeclareUnicodeCharacter{03B8}{$\theta$}
\DeclareUnicodeCharacter{03C9}{$\omega$}

\title{Organizing Physics with Open Energy-Driven Systems}
\author{
  Matteo Capucci
  \institute{University of Strathclyde}
  \email{matteo.capucci@gmail.com}
  \and
  Owen Lynch
  \institute{Topos Institute}
  \email{owen@topos.institute}
  \and
  David I.\ Spivak%
  \thanks{Spivak and Lynch's contribution was supported by the Air Force Office of Scientific Research under award number FA9550-23-1-0376.}
  \institute{Topos Institute}
  \email{david@topos.institute}
}
\date{}

\newcommand{\titlerunning}{Organizing Physics with Open Energy-Driven Systems}
\newcommand{\authorrunning}{M. Capucci, O. Lynch \& D.I. Spivak}

\hypersetup{
  bookmarksnumbered,
  pdftitle    = {\titlerunning},
  pdfauthor   = {\authorrunning},
  pdfsubject  = {EPTCS},
}

\RequirePackage{float}
\RequirePackage{graphicx}
\RequirePackage{subcaption}
\RequirePackage{quiver}
\RequirePackage{tikz}
\usetikzlibrary{
	angles,
	arrows,
	automata,
	backgrounds,
	calc,
	cd,
	decorations.markings,
	decorations.pathreplacing,
	decorations.text,
	decorations.pathmorphing,
	patterns,
	petri,
	positioning,
	quotes,
	shapes,
	snakes,
}

\RequirePackage[english]{babel}
\RequirePackage{csquotes}

\RequirePackage{amsfonts,amsmath,amssymb,amsthm}
\RequirePackage{mathrsfs,mathtools}
\RequirePackage{physics}
\RequirePackage{stmaryrd}

\hypersetup{
  bookmarks=true,
  pdfencoding=unicode,
  colorlinks=true,       
  linkcolor=red,        
  filecolor=blue,        
  urlcolor=blue,         
  citecolor=green,        
  pdfborder={0 0 0}      
}
\RequirePackage[capitalize]{cleveref}

\RequirePackage{comment}

\RequirePackage{enumitem}
\RequirePackage{microtype}
\RequirePackage[all]{nowidow}

\theoremstyle{plain}
\newtheorem{theorem}{Theorem}

\newtheorem{proposition}[theorem]{Proposition}

\newtheorem{remark}[theorem]{Remark}

\theoremstyle{definition}
\newtheorem{construction}[theorem]{Construction}
\newtheorem{definition}[theorem]{Definition}
\newtheorem{example}[theorem]{Example}

\renewcommand*{\to}[1][]{\xrightarrow{#1}}

\newcommand*{\cat}[1]{\mathbf{#1}}
\newcommand*{\dbl}[1]{\mathbb{#1}}
\newcommand{\dblcat}[1]{\cat{\dbl #1}}
\newcommand*{\then}{\fatsemi}
\newcommand*{\id}{1}

\newcommand{\R}{\mathbb{R}}

\newcommand*{\Cat}{\dblcat{Cat}}
\newcommand*{\COrg}{\cat{COrg}}

\newcommand*{\Para}{\cat{Para}}

\newcommand*{\Set}{\cat{Set}}

\newcommand{\bigarena}[2]{
     \begin{pmatrix}{\vphantom{f_f^f}#2} \\ {\vphantom{f_f^f}#1} \end{pmatrix}
}
\newcommand{\littlearena}[2]{
     \begin{psmallmatrix}{\vphantom{f}#2} \\ {\vphantom{f}#1} \end{psmallmatrix}
}
\newcommand{\arena}[2]{
  \relax\if@display
     \bigarena{#1}{#2}
  \else
     \littlearena{#1}{#2}
  \fi
}

\newcommand{\lens}[2]{\arena{#2}{#1}}

\newcommand{\To}[1]{\xrightarrow{#1}}
\newcommand{\longto}{\longrightarrow}
\newcommand{\lensto}{\leftrightarrows}

\renewcommand*{\real}{\mathbb{R}}

\newenvironment{eqalign}{\begin{equation}\begin{aligned}}{\end{aligned}\end{equation}}
\newenvironment{eqalign*}{\begin{equation*}\begin{aligned}}{\end{aligned}\end{equation*}}

\tikzset{
  relation/.style={
    draw=none,
    every to/.append style={
      edge node={node [sloped, allow upside down, auto=false]{$#1$}}}
  }
}

\tikzcdset{
  mono/.code={
    \pgfsetarrows{tikzcd to reversed-tikzcd to}
  }
}
\tikzcdset{
  epi/.code={
    \pgfsetarrowsend{tikzcd double to}
  }
}
\tikzcdset{
  into/.code={
    \pgfsetarrows{tikzcd right hook-tikzcd to}
  }
}
\tikzcdset{
  twocell/.style={Rightarrow, shorten >= 3ex, shorten <= 3ex}
}

\tikzcdset{
  row sep/normal={
    6ex
  },
  column sep/normal={
    8ex
  }
}

\mathchardef\dash="2D

\newcommand{\epito}{\twoheadrightarrow}

\newcommand{\nlongto}[1]{\xrightarrow{\;#1\;}}

\newcommand{\nisoto}[1]{\xrightarrow[\sim]{#1}}
\newcommand{\nisofrom}[1]{\xleftarrow[\sim]{#1}}
\newcommand{\isofrom}[1]{\nisofrom{}}
\newcommand{\isoto}[1]{\nisoto{}}

\usepackage{stackengine}

\newcommand{\iso}[1][]{\overset{#1}{\cong}}

\newcommand{\Smooth}{\cat{Mfd}}
\newcommand{\Smoothiso}{\Smooth^{\textnormal{iso}}}

\newcommand{\Org}{\cat{Org}}
\newcommand{\dblOrg}{\dblcat{Org}}

\newcommand{\systh}{\mathbf}

\newcommand{\Sys}{\systh{Sys}}

\newcommand{\Lens}{\cat{Lens}}

\newcommand{\ODE}{\mathbf{ODE}}
\newcommand{\Reactions}{\mathbf{React}}

\newcommand{\Subm}{\mathbf{Subm}}
\newcommand{\Poly}{\mathbf{Poly}}
\newcommand{\Coalg}{\mathbf{Coalg}}

\newcommand{\OpenReact}{\cat{OpenReact}}
\newcommand{\OpenPot}{\cat{OpenErg}}
\newcommand{\IntReact}{{\int\Reactions}}
\newcommand{\OpenODE}{\cat{OpenODE}}

\newcommand{\ActionODE}{\otimes_\ODE}
\newcommand{\ActionT}{\otimes_T}

\renewcommand{\op}{{\mathsf{op}}}

\newcommand{\SmoothLens}{{\Smooth\Lens}}

\tikzstyle{midarrow}=[decoration={
                            markings,
                            mark=at position 0.5 with {\arrow{>}}},postaction={decorate}]

\tikzstyle{midarrowflip}=[decoration={
                            markings,
                            mark=at position 0.5 with {\arrow{<}}},postaction={decorate}]

\begin{document}

\maketitle

\begin{abstract}
    Organizing physics has been a long-standing preoccupation of applied category theory, going back at least to Lawvere. We contribute to this research thread by noticing that Hamiltonian mechanics and gradient descent depend crucially on a consistent choice of transformation---which we call a reaction structure---from the cotangent bundle to the tangent bundle. We then construct a compositional theory of reaction structures. Reaction-based systems offer a different perspective on composition in physics than port-Hamiltonian systems or open classical mechanics, in that reaction-based composition does not create any new constraints that must be solved for algebraically.

    The technical contributions of this paper are the development of symmetric monoidal categories of open energy-driven systems and open differential equations, and a functor between them, functioning as a ``functorial semantics'' for reaction structures. This approach echoes what has previously been done for open games and open gradient-based learners, and in fact subsumes the latter. We then illustrate our theory by constructing an $n$-fold pendulum as a composite of $n$-many pendula.
\end{abstract}

\section{Introduction}

A long-standing goal of applied category theory is to provide compositional frameworks for physical systems \cite{lawvere_Toward_1980}. In this paper, we introduce a compositional framework for a generalization of both systems based on Hamiltonian mechanics and general systems which perform gradient descent. In this framework, systems interact by sending each other gradients. These gradients are converted to motion (that is, tangent vectors) by what we call a \emph{reaction}. In Hamiltonian dynamics, the reaction arises from the symplectic structure and in gradient descent, the reaction arises from a Riemannian structure.

In~\cref{sec:intuition} we explain the basic intuition: at the core, Hamiltonian dynamics and gradient descent on a manifold $X$ both depend on a map $R\colon T^*X\to TX$, which translates gradients into motion. One can imagine this as a functional version of argmax, in that given a ``valuation'' function $\varphi \colon T_x X\to\real$, one obtains a choice of element $R(\varphi) \in T_x X$. But whereas gradient descent really does choose a tangent direction that maximizes change along the gradient, Hamiltonian dynamics does almost the opposite, choosing a tangent direction with zero change along the gradient.

Our compositional framework is bidirectional but oriented. This means that some parts of an open system are deemed inputs whereas other parts are deemed outputs. For example, in a pendulum, the input is the position and momentum of the pivot point, and the output is the position and momentum of the bob, which another system may take as input. However, the system also \emph{receives} gradients on its output, and \emph{sends} gradients to its input, which can be though as the result of external forces. It also keeps and updates an internal state. This can be pictured in the following way.

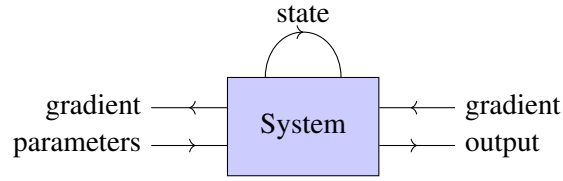
\[
  \begin{tikzpicture}
  \tikzstyle{box}=[draw, fill=blue!20, minimum width=2cm, minimum height=1.3cm, align=center]

  \node[box] (main) {System};

  \draw[midarrow] ([yshift=0.25cm]main.west) -- ++(-1,0) node[left] {gradient};
  \draw[midarrowflip] ([yshift=-0.25cm]main.west) -- ++(-1,0) node[left] {parameters};
  \draw[midarrowflip] ([yshift=0.25cm]main.east) -- ++(1,0) node[right] {gradient};
  \draw[midarrow] ([yshift=-0.25cm]main.east) -- ++(1,0) node[right] {output};
  \draw[midarrow] ([xshift=-0.5cm]main.north) to [out=90,in=180] (0,1.25cm)  node[above] {state} to [out=0,in=90] ([xshift=0.5cm]main.north);
  \end{tikzpicture}
\]

We begin in~\cref{sec:intuition} by reviewing Hamiltonian and gradient descent systems and their commonalities, motivating the ensuing developments.
Then we define symmetric monoidal category of open energy-driven systems in~\cref{sec:open-pot-cat}; there the bidirectionality is not explicit, but instead there is an explicit choice of reaction $R$. In~\cref{sec:semantics} we define a functor that assigns these systems a bidirectional semantics in terms of parametric lenses. We do so in two steps, the first of which is more likely generalizable and the second of which is conceptually simpler. These semantics land in lens categories, which are now well-known in categorical machine learning literature. In fact our semantics in inspired by~\cite{fong_Backprop_2019, cruttwell_Categorical_2022}, which we extend to general smooth manifolds.

Finally, in~\cref{sec:examples} we give one example of these ideas, that of an $n$-fold pendulum constructed by composing single pendula.
We conclude with an epilogue in~\cref{sec:epilogue}.

\subsection{Related work}

This framework is connected to a variety of other attempts to formalize physical systems within applied category theory.

\begin{itemize}[nosep]
  \item The title of this paper is a pun based on the double category $\dblOrg$ that was first developed by the third-named author in \cite{spivak_Learners_2022,shapiro_Dynamic_2023}. One way of thinking about this paper is that it develops a ``continuous version of $\dblOrg$.''
  \item Another approach to ``systems exerting force on each other'' is port-Hamiltonian systems, which has been developed categorically by the second-named author in \cite{lynch_Relational_2022}, \cite{lohmayer_Exergetic_2024}. However, the doctrine of composition developed for port-Hamiltonian systems is undirected and relational, and so a computer implementation of this composition would require solving differential-algebraic equations in a similar way to \cite{ma_ModelingToolkit_2021} rather than just differential equations. In contrast, the framework in the current paper gives an ``input-output'' view on physical systems, and thus composition does not introduce new constraints that must be solved for. In future work, we hope to give an account of the relationship between the directed and undirected accounts of composition.
  \item Classical mechanics has been previously treated from a category theoretic viewpoint via spans of symplectic manifolds \cite{baez_Open_2021}. However, like port-Hamiltonian systems, this approach is essentially relational, necessitating semantics in differential-algebraic equations.
  \item We heavily rely on the Para construction and its functoriality, as well as using roughly the same techniques of~\cite{capucci_Towards_2022} for constructing our symmetric monoidal categories of open systems.
  \item The pattern of having a simple description of feedback systems that then gets compiled down to parametric lenses of some sort follows what~\cite{capucci_Diegetic_2023} has shown for open games. The analogy with gradient-based systems, formulated in terms of changes and valuations, already appears in \emph{ibid.}
  \item Resource sharing machines are another method of composing dynamical systems which has been applied to dynamical systems for physics \cite{libkind_Operadic_2022}. In contrast with the present approach, resource sharing machines do not derive their dynamics from potentials and forces; rather a resource sharing machine takes the vector field as primitive. However, we hope that in the future ``resource-sharing composition'' will be available to use with the formalism of this paper.
\end{itemize}

\subsection{Notational conventions}
For composition, we use $\then$ to denote diagrammatic order.
For a vector bundle $\pi\colon E\to B$, we denote the set of global sections by $\Gamma(\pi)\coloneqq\{s\colon B\to E\mid s\then\pi=\id\}$. Given another vector bundle $\pi'\colon E'\to B$, we denote the vector bundle (over $B$) of fiberwise linear maps between them by $[\pi,\pi']$.

We denote by $(\Smooth,\real^0,\times)$ the cartesian monoidal category of smooth manifolds and smooth maps between them, although everything we're saying works in well-behaved generalizations, such as diffeological spaces \cite{souriau_Groupes_2006}. We denote the circle by $S^1\in\Smooth$. The wide subcategory of smooth manifolds and isomorphisms between them is denoted $\Smoothiso$; from which it inherits the (no longer cartesian) monoidal structure $(\real^0,\times)$.

\subsection{Prerequisites}
In order for the framework of this paper to be comprehensible, a certain amount of context must be given, but also a certain amount of context must be omitted for brevity.

The context that we omit and we assume the reader to be already acquainted with is basic differential geometry (definition and strong monoidal functoriality of tangent and cotangent bundles), and especially the theory of vector bundles (we reference~\cite{kolar_Natural_1993}).
We also expect the reader to be familiar with the $\Para$ construction (see~\cite{fong_Backprop_2019,capucci_Towards_2022,cruttwell_Categorical_2022}) and its functorial properties, but we nonetheless spell out the result when we invoke it.

It will also help the reader to be familiar with Hamiltonian mechanics and gradient descent, but we will briefly review these so that a sufficiently determined reader may get through this paper without too much prior experience.

Lastly, while we kept this work strictly in the land of 1-categories for brevity, we use ideas from categorical systems theory and think of our constructions as shadows of their essentially double-categorical nature.
So being aware of~\cite{myers_Categorical_2023} is not required (except in passages where we explictly draw a connection) but might help understanding the subtext.

\section{Intuition}\label{sec:intuition}

Hamiltonian mechanics and gradient descent have a common mathematical structure. We start by reviewing them.

\subsection{Hamiltonian mechanics} \label{sec:hamiltonian_mechanics}
A Hamiltonian system consists of a state space $X \in \Smooth$ with a full-rank symplectic form $\omega \in \Gamma(T^\ast X \wedge T^\ast X)$ and a function $H \colon X \to \real$, called the \emph{Hamiltonian}, which represents the system's energy at any $x\in X$. Applying $\omega$ and the sequence of maps
\[ T^\ast X \wedge T^\ast X \to T^\ast X \otimes T^\ast X \cong [TX, T^\ast X] \]
we produce a section $K_\omega \in \Gamma[TX, T^\ast X]$, i.e.~a linear map $TX\to T^\ast X$ over $X$, which is invertible because $\omega$ is full-rank. Call its inverse $J_\omega \in \Gamma[T^\ast X, TX]$; we refer to it as the \emph{reaction} associated to $\omega$. Then the dynamics of the system $(X, \omega, H)$ are given by Hamilton's equation, i.e.~the differential equation
\begin{equation}\label{eq:hamiltons}
  {\dv{x}{t}}(t) = J_\omega(x(t)) \; \dd H(x(t)).
\end{equation}

Note that the dynamics only depend on the reaction $J_\omega$ rather than $\omega$; we started with the symplectic form to connect to the more conventional way of doing things, but we will almost exclusively only work with reactions in the future.

\begin{remark}
  Notice how we only need $\omega$ to be non-degenerate because we want to take the inverse of $K_{\omega}$. However, if we had just started from $J_\omega$, we need not assume that $J_\omega$ is invertible.
  In fact we get a reaction also from a \emph{Poisson structure}~\cite{crainic_lectures_2021}, which boils down to a possibly-degenerate $J_{\omega}$.
  These can be used to model Hamiltonian systems with odd dimension.
\end{remark}

For any smooth manifold $M\in\Smooth$ of dimension $n$, there is a canonical symplectic structure on its cotangent space $X\coloneqq T^*M$. Given a coordinate chart for $M$ and the induced coordinate chart on $T^*M$ and $TM$, the corresponding reaction $J\colon T^*_xX\to T_xX$ over $x\in X$ has the following form
\begin{equation}\label{eq:reaction_cotan}
J(x)=\begin{bmatrix}0&I_n\\-I_n&0\end{bmatrix}
\end{equation}
where $I_n$ is the $(n\times n)$-identity matrix. Note that $J(x)$ is independent of the choice of coordinate chart \cite[Chapter~8]{arnold_Mathematical_2013}.

\begin{example} \label{ex:pendulum}
  Consider the pendulum of fixed length $l$ and mass $m$ shown here
  \begin{equation}\label{eq:pendulum}
    \scalebox{0.85}
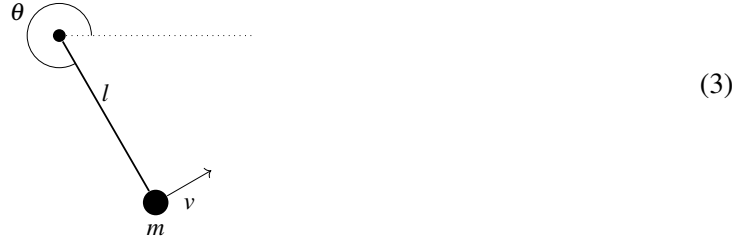
{\begin{tikzpicture}[baseline=(ell)]
      \node[inner sep = 2pt, fill, black, circle] (x0) at (0, 0) {};
      \node[inner sep = 4pt, fill, black, circle, label=below:$m$] (x1) at (300:3cm) {};
      \coordinate (xaxis) at (3, 0);
      \draw[thick] (x0) -- node[above, outer sep=3pt] (ell) {$l$} (x1);
      \draw[dotted] (x0) -- (xaxis);
      \draw (0.5, 0) arc (0:300:0.5cm);
      \draw[->] (x1) -- node[below, outer sep=3pt] {$v$} +(30:1cm);
      \node at (150:0.75cm) {$\theta$};
    \end{tikzpicture}}
  \end{equation}
  The little dot is called the \emph{pivot} and the big dot is called the \emph{bob}. We can model this as a Hamiltonian system by letting $X = T^\ast S^1$ with coordinates $(\theta, L)$ representing angle and angular momentum. Then \eqref{eq:reaction_cotan} becomes
    \[ J(x) = \begin{bmatrix} 0 & 1 \\ -1 & 0 \end{bmatrix} \]
  because $S^1$ is 1-dimensional. For a given map $H \colon T^\ast S^1 \to \real$,~\cref{eq:hamiltons} reads

  \begin{eqalign*}
    \dv{\theta}{t} = \phantom{-}\pdv{H}{L}, \quad \dv{L}{t} = -\pdv{H}{\theta}
  \end{eqalign*}

  To model a pendulum as in \eqref{eq:pendulum} with kinetic and gravitational energy, we would make the following Hamiltonian. First, we can compute several derived quantities from the variables $(\theta, L)$ of the system (the following quantities are all implicit functions of $(\theta, L)$).
  \begin{eqalign*}
    I &= m l^2 & \text{(moment of inertia)} \\
    \omega &= \frac{L}{I} & \text{(rotational velocity)} \\
    x &= l \begin{bmatrix} \cos \theta \\ \sin \theta \end{bmatrix} & \text{(position of the mass)} \\
    v &= l \omega \begin{bmatrix} -\sin \theta \\ \cos \theta \end{bmatrix} & \text{(velocity of the mass)} \\
    h &= x_2=l\sin\theta & \text{(height of the mass)}
  \end{eqalign*}
  The Hamiltonian is then the sum of kinetic and gravitational energy, written as
  \[ H(\theta, L) = \frac{1}{2}m \norm{v}^2 + m g h \]
  where $g$ is gravitational acceleration. Plugging this $H$ into~\cref{eq:hamiltons} gives

  \begin{eqalign*}
    \dv{\theta}{t} &= \frac{L}{I} \\
    \dv{L}{t} &= -mgl\cos(\theta)
  \end{eqalign*}
\end{example}

\begin{definition}\label{def.react}
  For any manifold $X\in\Smooth$, a \emph{reaction} on $X$ is a map $J\colon T^*X\to TX$ over $X$. The set of reactions on $X$ is denoted $\Reactions(X)$.
  Given a smooth map $f\colon X \to Y$, we denote by $\Reactions(f)\colon\Reactions(X)\to\Reactions(Y)$ the map sending $J$ to the composite $T^*Y \To{T^*f} T^*X \To{J} TX \To{Tf} TY$.
\end{definition}

\subsection{Gradient descent}
Gradient descent (or gradient ascent) is a very similar story. We start with a state space $X$ along with a Riemmanian metric $g \in \Gamma(T^\ast X \otimes T^\ast X)$ and a function $S \colon X \to \real$. We apply the exact same procedure to $g$ that we did to $\omega$ in order to get a section $M \in \Reactions(X)$, and then we get the equation

\[ v(x) = M(x) \; \dd S(x) \]

\begin{example}
  Let $X = \real^2$ with the Euclidean metric. Then gradient ascent for a function $S \colon X \to \real$ looks like

  \begin{equation*}
    \dv{x_1}{t} = \pdv{S}{x_1}, \quad \dv{x_2}{t} = \pdv{S}{x_2}
  \end{equation*}

  It is more traditional to write this as $\dv{x}{t} = \grad S(x)^T$.
  It may seem like this doesn't use any fancy Riemannian structure, but in fact transposing a row vector into a column vector is made possible using the isomorphism $(\real^n)^\ast \cong \real^n$, the same isomorphism on which the natural inner product for $\real^n$ is built.
\end{example}

From Hamiltonian system and gradient descent systems we can provide the following common generalization:

\begin{definition}
  An \emph{energy-driven system} consists of a \emph{state space} $X \in \Smooth$, a \emph{reaction} $R \in \Reactions(X)$, and an \emph{energy functional} $E \colon X \to \real$.
\end{definition}

The reaction embodies the ``laws of physics'', turning energies (given by $E$) into forces.

\section{Open Energy-driven Systems} \label{sec:open-pot-cat}
Most, if not all, systems are in practice \emph{open}, meaning they are amenable to composition with other systems.

\begin{example} \label{ex:open_pendulum}
  Take, for instance, the pendulum system from~\cref{ex:pendulum}.
  In there we considered the pivot to be a fixed point, but in practice both pivot and bob are physical locations at which other systems can be attached.
  This means one can consider the pendulum as parametrized by $A = T\real^2$, the phase space of the pivot, and to influence the location and velocity of the bob by exposing such quantities in $B = T\real^2$.
  Again, $X = T^\ast S^1$, but now, rather than $E \colon X \to \real$,  we have $E \colon A \times X \to \real$ defined by
  \[  E((x_0, v_0), (\theta, L)) = \frac{1}{2m} \norm{v_0 + v}^2 + mg(h_0 + h) \]
  This $E$ represents the potential plus kinetic energy of a pendulum whose pivot is at position $x_0$ and moving at velocity $v_0$.

  We then define $w \colon T\real^2 \times T^*S^1 \to T\real^2$ by
  \[ w((x_0, v_0), (\theta, L)) = (x_0 + x, v_0 + v) \]
  This gives the position and velocity of the bob.
\end{example}

The ``open pendulum'' we just described is but an example of a general definition:

\begin{definition}\label{def.open_ed_system}
  Given $A,B \in \Smooth$, an \emph{open energy-driven system with inputs $A$ and outputs $B$} consists of a manifold $X$, a reaction $R \in \Gamma[T^\ast X, TX]$, a function $E \colon A \times X \to \real$, and a function $w \colon A \times X \to B$. We call $A$ the \emph{inputs}, $X$ the \emph{state}, and $B$ the \emph{output}.
\end{definition}

The fact that $A = B$ in~\cref{ex:open_pendulum} is tantalizing, because it makes the pendulum something like a ``parametric endomorphism'' of $T\real^2$. Could we somehow ``compose the pendulum with itself'' to make a double pendulum? Also, how can we describe the differential equation attached to an open energy-driven system; i.e.~what is the semantics of an open energy-driven system? Answering these questions is the subject of the remainder of this paper.

\subsection{Composing Open Energy-driven Systems}

The purpose of this section is to develop a symmetric monoidal category where the morphisms are open energy-driven systems.

We begin with the following straightforward proposition.

\begin{proposition}\label{prop:reactions-fun}
  The functor $\Reactions\colon\Smoothiso\to\Set$ from~\cref{def.react} is lax symmetric monoidal, with unitor $1\to\Reactions(1)$ given by the unique element of $\Reactions(1)$ (the zero map $\real^0 \to \real^0$) and compositor
  \begin{equation}
    \oplus_{X_1,X_2} \colon \Reactions(X_1) \times \Reactions(X_2) \to \Reactions(X_1 \times X_2)
  \end{equation}
  defined by sending $R_1 \in \Reactions(X_1), R_2 \in \Reactions(X_2)$ to
  \begin{equation*}
   T_{(x_1,x_2)}^\ast (X_1{\times} X_2) \cong T^\ast_{x_1} X_1 \times T^\ast_{x_2} X_2 \To{R_1(x_1) \times R_2(x_2)} T_{x_1} X_1 \times T_{x_1} X_2 \cong T_{(x_1, x_2)}(X_1 {\times} X_2). \
   \end{equation*}
\end{proposition}

\begin{definition}
  Let $\IntReact \To{\pi_\Reactions}\Smoothiso$ be the monoidal Grothendieck construction of $\Reactions$ (see \cite{moeller_Monoidal_2020}).
  An object of $\IntReact$ is a pair $(X,R)$ consisting of a space and a reaction on it, a morphism is a diffeomorphism which preserves the reactions, and the monoidal product $(X_1,R_1)\otimes(X_2,R_2)$ is given by $(X_1\times X_2,R_1\oplus R_2)$ as above.
\end{definition}

\begin{example}
  The symmetric monoidal category of symplectic manifolds and symplectomorphisms embeds faithfully into $\IntReact$ by sending a symplectic manifold $(X, \omega)$ to $(X, J_\omega)$, where $J_\omega$ is the reaction associated to $\omega$, as constructed in~\cref{sec:hamiltonian_mechanics}.
  A \emph{symplectomorphism} is a diffeomorphism that preserves the symplectic structure, and it is not too hard to show that preserving the symplectic structure implies preserving the reaction structure.
\end{example}

Now, as the composite $\IntReact \to[\pi_\Reactions] \Smoothiso \to[U] \Smooth$ is strong monoidal, we have a monoidal action of $\IntReact$ on $\Smooth$, simply given by $(X,R) \otimes_U Y\coloneqq X \times Y$.
The technique of decorating an action with data coming from a monoidal Grothendieck construction was introduced by~\cite{capucci_Towards_2022}.

We can then construct a category $\OpenReact$ by applying the Para construction \cite{fong_Backprop_2019,capucci_Towards_2022} to $\otimes_U$, and in fact we use this opportunity to recall how the latter is performed.

\begin{definition}
  Let $\OpenReact$ be the symmetric monoidal category of \emph{open reaction systems} defined as the local 0-truncation\footnote{Taking the \textbf{local 0-truncation} of a bicategory means we replace each hom-\emph{category} with the \emph{set} of its isomorphism classes.
  Notice this yields a well-defined strict 1-category since coherence isomorphisms are turned into equalities.} of $\Para(\otimes_U)$.
\end{definition}

Concretely, an object in $\OpenReact$ is a manifold, a 1-cell $A\to B$ is an open reaction system and is given by a pair $(X,R)\in\IntReact$ together with a smooth map $w\colon X\times A\to B$.
These 1-cells compose by accumulating parameters if $(X_1,R_1,w_1:A \times X_1 \to B)$ and $(X_2,R_2,w_2:B \times X_2 \to C)$ are composable 1-cells, their composite is:
\begin{equation}
  (X_1,R_1,w_1) \then (X_2,R_2,w_2) := (X_1 \times X_2, R_1 \oplus R_2, A \times X_1 \times X_2 \To{w_1 \times X_2} B \times X_2 \To{w_2} C).
\end{equation}
Morphisms $(X,R,w)$ and $(X',R',w')$ are considered equivalent if there is an isomorphism $i\colon X\iso X'$ with $\Reactions(i)(R)=R'$ and $w=(i\times A)\then w'$.
The symmetric monoidal structure is the same as $\Smooth$ on objects, while on 1-cells $(X,R,w):A \to B$ and $(Y,S,v):C \to D$ is given by
\begin{equation*}
  (X, R, w) \otimes (Y,S,v) := (X \times Y, R \oplus S, (A \times C) \times (X \times Y) \To{\sim} (X \times A) \times (Y \times C) \To{w \times v} B \times D).
\end{equation*}
The fact that this symmetric monoidal structure is well-defined on the local 0-truncation is proven in~\cite{hermida_monoidal_2012}.

Compared to energy-driven systems, reaction systems lack the data of an energy functional.
We add this by considering such a functional as an effect. In fact $(\real,0,+)$ is a monoid in $\OpenReact$, because it is a monoid in $\Smooth$ and $\Smooth$ embeds into $\OpenReact$, and thus $(-) \times \real$ is a monad on $\OpenReact$.

\begin{definition}
  Let $\OpenPot$ be the symmetric monoidal category of \emph{open energy-driven systems}, defined as the Kleisli category of $(-) \times \real$ on $\OpenReact$.
\end{definition}

A 1-cell in $\OpenPot$ from $A$ to $B$ is then a state space $X$ with a reaction $R \in \Reactions(X)$ and a smooth function $\langle w, E \rangle : A \times X \to B \times \real$, which is precisely an open energy-driven system as in~\cref{def.open_ed_system}.

The composition of $(X,R, \langle w, E \rangle:A \times X \to B \times \R)$ and $(X',R', \langle w', E' \rangle:B \times X' \to C \times \R)$ is the energy-driven system:
\begin{equation}
  (X \times X', R \oplus R', \langle w \then w', E + w^*E' \rangle :A \times (X \times X') \to C \times \R)
\end{equation}
where $w^*E' : A \times (X \times X') \to \R$ is given by $(a,x,x') \mapsto E'(w(a,x),x')$ and we abuse notation by writing $E$ for $\pi_{A,X} \then E$.
The symmetric monoidal structure is defined similarly as above.

\begin{remark}\label{rmk:dep-para}
  There are more natural higher-category structures that we could use instead of just symmetric monoidal categories. For instance, the Para construction naturally produces a double category.
  In fact the generalized Para construction of~\cite{myers_Para_2023,capucci_Constructing_2023} can even define $\OpenPot$ in one fell swoop by having the component $E:A \times X \to \real$ be part of the decoration on the states.
\end{remark}

\section{The Semantics of Open Energy-driven Systems}\label{sec:semantics}

In order to develop semantics for open energy-driven systems, we must have some notion of an open ODE. It is the job of the current section to construct this.

\subsection{Open Ordinary Differential Equations}

\begin{definition}
  A smooth function $p \colon \bar X \epito X$ is a \emph{submersion} if all $(Tp)_{\bar x} \colon T_{\bar x} \bar X \to T_{p(\bar x)} X$ are surjective.
\end{definition}

\begin{proposition}
  If $p \colon \bar X \epito X$ is a submersion, then all pullbacks along $p$ exist, and are computed as in the category of topological spaces.
\end{proposition}

\begin{proof}
  Standard, can be found in \cite[Corollary~I.2.19]{kolar_Natural_1993}.
\end{proof}

\begin{definition}
  Let $\Subm \colon \Smooth^\op \to \Cat$ be the pseudofunctor which sends a space $X$ to the category of submersions $\bar{X}\to X$ over it, and a smooth function $f \colon X \to Y$ to its action by pullback $f^\ast \colon \Subm(Y) \to \Subm(X)$.
\end{definition}

\begin{definition}
  Let $\SmoothLens$ be the symmetric monoidal category of \textbf{$\Subm$-lenses}, defined in~\cite{spivak_Generalized_2019} as $\int \Subm^\op$, meaning its objects are given by submersions $p \colon \bar{X} \epito X$ and denoted by $\lens{\bar{X}}{X}$ and its 1-cells, denoted as below left, are given by pairs of dashed arrows as below right:
  \begin{equation}
  \lens{f^\sharp}{f} : \lens{\bar X}{X} \lensto \lens{\bar Y}{Y}
  \qquad
  \begin{tikzcd}[ampersand replacement=\&,sep=scriptsize]
    {\bar X} \& {f^*\bar Y} \& {\bar Y} \\
    X \& X \& Y
    \arrow["f"', dashed, from=2-2, to=2-3]
    \arrow["p"', two heads, from=1-1, to=2-1]
    \arrow["q", two heads, from=1-3, to=2-3]
    \arrow[Rightarrow, no head, from=2-1, to=2-2]
    \arrow[two heads, from=1-2, to=2-2]
    \arrow[from=1-2, to=1-3]
    \arrow["\lrcorner"{anchor=center, pos=0.125}, draw=none, from=1-2, to=2-3]
    \arrow["{f^\sharp}"', dashed, from=1-2, to=1-1]
  \end{tikzcd}
  \end{equation}
  The symmetric monoidal product of this category is given by fibrewise product of bundles.
\end{definition}

The reader might notice that the form of these $\Subm$-lenses is exactly that of the pullback map on covectors induced between the cotangent bundles by any smooth map, i.e.~if $f:X \to Y$ is a smooth map there is a corresponding lens
\begin{equation*}
  \begin{tikzcd}[ampersand replacement=\&,sep=scriptsize]
    {T^*X} \& {f^*T^*Y} \& {T^*Y} \\
    X \& X \& Y
    \arrow["f"', dashed, from=2-2, to=2-3]
    \arrow["p"', two heads, from=1-1, to=2-1]
    \arrow["q", two heads, from=1-3, to=2-3]
    \arrow[Rightarrow, no head, from=2-1, to=2-2]
    \arrow[two heads, from=1-2, to=2-2]
    \arrow[from=1-2, to=1-3]
    \arrow["\lrcorner"{anchor=center, pos=0.125}, draw=none, from=1-2, to=2-3]
    \arrow["{T^*f}"', dashed, from=1-2, to=1-1]
  \end{tikzcd}
\end{equation*}
In fact the assignment $f \mapsto \lens{T^*f}{f}$ defines a functor $T^*: \Smooth \to \SmoothLens$, which will be central later.

We want to interpret an open energy-driven system on $A, B$ with state space $X$ as a parametric lens:
\[ \lens{T^\ast A}{A} \otimes \lens{T^* X}{X} \lensto \lens{T^\ast B}{B} \]
where the parameter space $X$ is decorated by the data of a reaction, seen as an open ODE $\lens{TX}{X} \lensto \lens{T^*X}{X}$.
We do this with a similar Para construction as before, though to tame coherence we directly construct the (op)fibration of $\pi_\ODE:\ODE \to \SmoothLens$ by means of an isocomma construction:

\begin{construction}\label{const:ode}
  Let $T : \Smoothiso \to \SmoothLens$ be the tangent bundle functor restricted to the core of $\Smooth$, i.e.~its wide subcategory of isomorphisms.
  Consider the isocomma $\ODE := \{T \cong \SmoothLens\}$, performed in the 2-category of symmetric monoidal categories and symmetric strong monoidal functors (which exists well-defined by e.g.~\cite{blackwell_two-dimensional_1989}).
  This is itself a symmetric monoidal category, opfibred over $\SmoothLens$, with fibers given by the groupoids:
  \begin{equation}
      \ODE\lens{\bar Y}{Y} = \left\{ X \in \Smooth, \lens{u^\sharp}{u} : \lens{TX}{X} \lensto \lens{\bar Y}{Y} \right\},
  \end{equation}
  with morphisms given by isomorphisms $\varphi:X \iso X'$ between the state spaces such that $\lens{T\varphi^{-1}}{\varphi} : \lens{TX}{X} \lensto \lens{TX'}{X'}$ commutes with the dynamics on $X$ and $X'$.
  The symmetric monoidal structure is inherited from that of $\SmoothLens$, and relies crucially on the fact that $T : \Smoothiso \to \SmoothLens$ is a strong monoidal functor.
\end{construction}


The projection $\pi_\ODE:\ODE \to \SmoothLens$ induces an action of the former on the latter, which we denote $\ActionODE$.

\begin{definition}
  We call $\OpenODE$ the local 0-truncation of $\Para(\ActionODE)$.
\end{definition}

Thus a map $\lens{\bar A}{A} \to \lens{\bar B}{B}$ in $\OpenODE$ is a choice of parameter interface $\lens{\bar P}{P}$ and of an open ODE over it, say $\lens{u^\sharp}{u} : \lens{TX}{X} \lensto \lens{\bar P}{P}$, and then a choice of smooth lens $\lens{\bar A}{A} \otimes \lens{\bar P}{P} \lensto \lens{\bar B}{B}$.
At this stage, the open ODE on the parameter is nothing more than a decoration.

\subsection{An ``organized'' view}
In the symmetric monoidal category $\OpenODE$, systems (the ODEs) and their wiring (the parametric lens they are grafted on) are kept neatly separated.
This is because, in general, the systems might be very different from their wiring, but in this case they aren't: both open ODEs and their wiring are differential lenses, so one can collapse the data of the ODE directly in the wiring, seeing it as all part of a unique process.
This brings us to define $\COrg$, which is a smooth space, continuous time variant of $\dblOrg$, introduced in~\cite{spivak_Learners_2022}; albeit we confine ourselves to defining a 1-category and not a full double categiry like $\dblOrg$ is.

Once again, we define an action through a symmetric monoidal functor, this time given by the tangent bundle functor $T\colon\Smoothiso\to\SmoothLens$ sending $X\mapsto\lens{TX}{X}$. This defines an action of $\Smoothiso$ on $\SmoothLens$, which we denote $\ActionT$.

\begin{definition}
  We call $\COrg$ the local 0-truncation of $\Para(\ActionT)$.
\end{definition}

The categories $\OpenODE$ and $\COrg$ have the same objects, namely submersions $\lens{\bar{A}}{A}$, but a 1-cell from $\lens{\bar{A}}{A}$ to $\lens{\bar{B}}{B}$ in $\COrg$ consists of a manifold $X$ (up to diffeomorphism) and a lens $\lens{\bar{A}}{A}\otimes\lens{TX}{X}\lensto\lens{\bar{B}}{B}$.

\newcommand{\collapse}{{\mathsf{collapse}}}
\begin{proposition}\label{prop:collapse}
  There is an identity-on-objects, symmetric monoidal functor
  \begin{equation*}
    \collapse : \OpenODE \longto \COrg
  \end{equation*}
  given by
  \begin{equation*}
    \begin{tikzcd}[ampersand replacement=\&,row sep=3.5ex]
      \&[-7.5ex]{\lens{TX}{X}} \\
      {\lens{\bar{A}}{A}} \& {\lens{\bar P}{P}} \& {\lens{\bar{B}}{B}}
      \arrow["\otimes"{marking, allow upside down}, draw=none, from=2-1, to=2-2]
      \arrow["f"', shift right, from=2-2, to=2-3]
      \arrow["{f^\sharp}"', shift right, from=2-3, to=2-2]
      \arrow[shift right, "u"', from=1-2, to=2-2]
      \arrow["u^\sharp"', shift right, from=2-2, to=1-2]
    \end{tikzcd}
    \mapsto
    \begin{tikzcd}[ampersand replacement=\&,row sep=-2ex]
      {\lens{TX}{X}} \& {\lens{\bar P}{P}} \\
      \&\& {\lens{\bar{B}}{B}} \\
      {\lens{\bar{A}}{A}} \& {\lens{\bar{A}}{A}}
      \arrow[""{name=0, anchor=center, inner sep=0}, "\otimes"{marking, allow upside down}, draw=none, from=3-2, to=1-2]
      \arrow["\otimes"{marking, allow upside down}, draw=none, from=3-1, to=1-1]
      \arrow[Rightarrow, no head, from=3-1, to=3-2]
      \arrow["u"', shift right, from=1-1, to=1-2]
      \arrow["{u^\sharp}"', shift right, from=1-2, to=1-1]
      \arrow["{{f^\sharp}}"'{pos=0.3}, shift right, shorten >=10pt, from=2-3, to=0]
      \arrow["f"'{pos=0.7}, shift right, shorten <=10pt, from=0, to=2-3]
    \end{tikzcd}
  \end{equation*}
\end{proposition}
\begin{proof}
  The fact this is functorial and symmetric monoidal boils down to the interchange law of the monoidal structure of $\SmoothLens$, which allows to map composition in $\OpenODE$ (which puts the ODEs in the parameter side by side) to composition in $\COrg$ (which composes sequentially the 1-cells).
\end{proof}

In fact, this definition applies already to discrete systems, i.e.~to (the horizontal 1-category of) $\dblOrg$, which we denote as $\Org$.
Recall this 1-category is the symmetric monoidal category whose objects are polynomial functors $p \in \Poly$ and whose 1-cells $p \to q$ are parametric maps of polynomials (i.e.~parametric dependent lenses) of the form $p \otimes Sy^S \to q$, thus where the parameter is given by something of the form $Sy^S$ (\emph{state systems} in the language of~\cite{niu_polynomial_2023}), and these compose in a Para way.

Starting with $\Poly$, define a functor $\Coalg \colon \Poly \to \Set$ which sends a polynomial to its (large) set of coalgebras. Applying the Para construction to the forgetful functor $\int \Coalg \to \Poly$, we get something similar to $\OpenODE$, where a 1-cell from $p$ to $q$ is a polynomial $r$, a coalgebra $\delta:S \to r(S)$, and a poly map $\lens{f^\sharp}{f}:p \otimes r \to q$.
Now, the coalgebra $\delta$ is in fact equivalent to a map $\hat\delta : Sy^S\to r$, and thus we can define a symmetric monoidal functor $\Para(\otimes_\Coalg) \to \Org$ which collapses the pair $(\delta, \lens{f^\sharp}{f})$ to a single map $p \otimes Sy^S \to q$.

Generally speaking, this trick works whenever the indexed set of systems $\Sys : \cat C \to \Set$ is `free' in a specific sense: its elements are the objects of the slice $T/\cat C$, there T is a functor $T:\cat{States} \to \cat C$ which picks out ``state spaces'',\footnotemark~not unlike what we did in~\cref{const:ode}.%
\footnotetext{Being a ``state space'' is an attitude, not a formal mathematical concepts: any functor into $\cat C$ suffices.}
In that case we can always build a collapse functor $\Para(\otimes_\Sys) \to \Para(\otimes_T)$ by reproducing the construction of~\cref{prop:collapse}.
Systems which are given by slicing under a functor play a central role in~\cite{myers_Categorical_2023}, where Myers shows most theories of systems can be obtained in this way.








\subsection{The Cotangent Functor}
The aim of this section is to show how the cotangent functor $T^* : \Smooth \to \SmoothLens$ induces a functor $\OpenPot \longto \OpenODE$:
\begin{equation}\label{eq:semantics}
  \begin{tikzcd}[ampersand replacement=\&,row sep=1.8ex, column sep=4ex]
    \&[-2.5ex] R \&[-4ex] {\Reactions(X)} \&[-2ex] \\
    A \& X \&\& B \times \real
    \arrow["\times"{marking, allow upside down}, draw=none, from=2-1, to=2-2]
    \arrow["{{\langle w, E \rangle}}", from=2-2, to=2-4]
    \arrow["\in"{marking, allow upside down}, draw=none, from=1-2, to=1-3]
  \end{tikzcd}
  \mapsto
  \begin{tikzcd}[ampersand replacement=\&,row sep=3.5ex]
    \&[-7.5ex] {\lens{TX}{X}} \&[-8ex] {\ODE\lens{T^*X}{X}} \&[-2ex] \\
    {\lens{T^*A}{A}} \& {\lens{T^*X}{X}} \&\& {\lens{T^*B}{B}}
    \arrow["\times"{marking, allow upside down}, draw=none, from=2-1, to=2-2]
    \arrow["w"', shift right, from=2-2, to=2-4]
    \arrow["{T^*w + \dd E}"', shift right, from=2-4, to=2-2]
    \arrow[shift right, Rightarrow, no head, from=1-2, to=2-2]
    \arrow["R"', shift right, from=2-2, to=1-2]
    \arrow["\in"{marking, allow upside down}, draw=none, from=1-2, to=1-3]
  \end{tikzcd}
\end{equation}

The functor could be obtained in one single step using the technology of the generalized Para construction we hinted at in~\cref{rmk:dep-para} with which we can see both $\OpenPot$ and $\OpenODE$ as Para constructions and thus induce the desired functor by exhibiting one between the underlying fibred actions (see~\cite{capucci_Constructing_2023}).
Since we don't have the space to introduce this machinery, we just give a direct construction.

This would make the analogy with~\cite{capucci_Diegetic_2023} total since that is how the functor from game descriptions to parametric lenses is obtained. Notably, in this case $T^*$ is strong monoidal thus making the semantics of open energy-driven systems truly compositional.


\newcommand{\T}{{\mathbf T}}
\begin{theorem}
  The assignment defined in~\eqref{eq:semantics} is a well-defined symmetric monoidal functor
  \[\T^*:\OpenPot \longto \OpenODE.\]
\end{theorem}
\begin{proof}
  It is easy to see it sends identities to identities.
  Given composable open energy-driven systems $(X,R, \langle w, E \rangle:A \times X \to B \times \R)$ and $(X',R', \langle w', E' \rangle:B \times X' \to C \times \R)$, their images are the composable open ODEs
  \begin{equation*}
    \begin{tikzcd}[ampersand replacement=\&,row sep=3.5ex]
      \&[-7.5ex]{\lens{TX}{X}} \\
      {\lens{T^*A}{A}} \& {\lens{T^*X}{X}} \& {\lens{T^*B}{B}}
      \arrow["\otimes"{marking, allow upside down}, draw=none, from=2-1, to=2-2]
      \arrow["w"', shift right, from=2-2, to=2-3]
      \arrow["{{T^*w + \dd E}}"', shift right, from=2-3, to=2-2]
      \arrow[shift right, Rightarrow, no head, from=1-2, to=2-2]
      \arrow["R"', shift right, from=2-2, to=1-2]
    \end{tikzcd}
    {\Huge\then}
    \begin{tikzcd}[ampersand replacement=\&,row sep=3.5ex]
      \&[-7.5ex]{\lens{TX'}{X'}} \\
      {\lens{T^*B}{B}} \& {\lens{T^*X'}{X'}} \& {\lens{T^*C}{C}}
      \arrow["\otimes"{marking, allow upside down}, draw=none, from=2-1, to=2-2]
      \arrow["w'"', shift right, from=2-2, to=2-3]
      \arrow["{{T^*w' + \dd E'}}"', shift right, from=2-3, to=2-2]
      \arrow[shift right, Rightarrow, no head, from=1-2, to=2-2]
      \arrow["R'"', shift right, from=2-2, to=1-2]
    \end{tikzcd}
  \end{equation*}
  which reduces to
  \begin{equation*}
    \begin{tikzcd}[ampersand replacement=\&,row sep=3.5ex]
      \&[-7.5ex]{\lens{T(X \times X')}{X \times X'}} \\
      {\lens{T^*A}{A}} \& {\lens{T^*(X \times X')}{X \times X'}} \& {\lens{T^*B}{B} \otimes \lens{T^*X'}{X'}} \& {\lens{T^*C}{C}}
      \arrow["\otimes"{marking, allow upside down}, draw=none, from=2-1, to=2-2]
      \arrow["w"', shift right, from=2-2, to=2-3]
      \arrow["{{T^*w + \dd E}}"', shift right, from=2-3, to=2-2]
      \arrow[shift right, Rightarrow, no head, from=1-2, to=2-2]
      \arrow["{R \oplus R'}"', shift right, from=2-2, to=1-2]
      \arrow["{w'}"', shift right, from=2-3, to=2-4]
      \arrow["{T^*w' + \dd E'}"', shift right, from=2-4, to=2-3]
    \end{tikzcd}
  \end{equation*}
  Given $a \in A, x \in X, x' \in X'$, the composite backward map sends a covector $\alpha \in T^*C$ (we omit indexing of bundles for brevity) to
  \begin{equation*}
    T^*w(T^*w'(\alpha) + \dd E'(w(a, x), x')) + \dd E(a, x)
    =
    T^*(w \then w')(\alpha) + \dd (w^*E' + E)(a, x, x').
  \end{equation*}
  The latter expression equals the image of the composite energy-driven system $(X \times X', R \oplus R', \langle w \then w', E + w^*E' \rangle :A \times (X \times X') \to C \times \R)$:
  \begin{equation*}
    \begin{tikzcd}[ampersand replacement=\&,row sep=3.5ex]
      \&[-7.5ex]{\lens{T(X \times X')}{X \times X'}} \\
      {\lens{T^*A}{A}} \& {\lens{T^*(X \times X')}{(X \times X')}} \&\& {\lens{T^*C}{C}}
      \arrow["\otimes"{marking, allow upside down}, draw=none, from=2-1, to=2-2]
      \arrow["{w \then w'}"', shift right, from=2-2, to=2-4]
      \arrow["{{T^*(w \then w') + \dd (E + w^*E')}}"', shift right, from=2-4, to=2-2]
      \arrow[shift right, Rightarrow, no head, from=1-2, to=2-2]
      \arrow["{R \oplus R'}"', shift right, from=2-2, to=1-2]
    \end{tikzcd}
  \end{equation*}
  Preservation of symmetric monoidal structure is trivial, as it amounts to no more than the analogue structure on $T^*$.
\end{proof}

We can further collapse the description of the dynamics of the system by folding the ODE into the map itself, thus landing in $\COrg$:
\begin{equation*}
  \OpenPot \nlongto{\T^*} \OpenODE \nlongto{\collapse} \COrg.
\end{equation*}

\section{Examples}\label{sec:examples}

We motivated the last three sections by saying that we were going to compose the pendulum with itself to get an $n$-fold pendulum. We do this now.

\begin{example}
  For intuition, observe the following diagram of the double pendulum. We will work out carefully the composition of two pendulums to create a double pendulum, and leave the iterated composition to the reader.

  \[
    \scalebox{0.7}
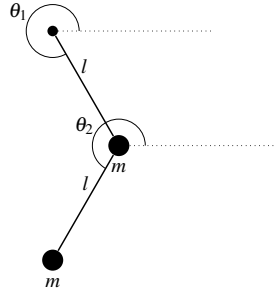
{\begin{tikzpicture}
      \node[inner sep = 2pt, fill, black, circle] (x0) at (0, 0) {};
      \node[inner sep = 4pt, fill, black, circle, label=below:$m$] (x1) at (300:2.5cm) {};
      \node[inner sep = 4pt, fill, black, circle, label=below:$m$] (x2) at ($(x1) + (240:2.5cm)$) {};
      \coordinate (xaxis) at (3, 0);
      \coordinate (xaxis2) at ($(x1) + (3, 0)$);
      \draw[thick] (x0) -- node[above, outer sep=3pt] {$l$} (x1);
      \draw[thick] (x1) -- node[above, outer sep=3pt] {$l$} (x2);
      \draw[dotted] (x0) -- (xaxis);
      \draw[dotted] (x1) -- (xaxis2);
      \draw (0.5, 0) arc (0:300:0.5cm);
      \draw ($(x1) + (0.5, 0)$) arc (0:240:0.5cm);
      \node at (150:0.75cm) {$\theta_1$};
      \node at ($(x1) + (150:0.75cm)$) {$\theta_2$};
    \end{tikzpicture}}
  \]

  Recall that the single pendulum in $\OpenPot$ as defined in~\cref{ex:open_pendulum} is an endomorphism $T \real^2 \to T \real^2$, whose state space is $T^\ast S^1$ and whose reaction structure is derived from the canonical symplectic structure and written in coordinate form as $J(x) = \begin{bmatrix} 0 & 1 \\ -1 & 0 \end{bmatrix}$.

  Composing this morphism with itself involves several steps, where we unwind all of the constructions that we've done so far. First of all, we have to take the monoidal product of $(T^\ast S^1, J)$ with itself in $\int \Reactions$. This uses the lax monoidal structure of $\Reactions$ to produce a reaction $J^{(2)} = J \oplus J$ on $T^\ast S^1 \times T^\ast S^1$. In coordinates, this is the block-diagonal matrix

  \[ J^{(2)}(x) = \begin{bmatrix} J(x) & 0 \\ 0 & J(x) \end{bmatrix} \]

  Transforming this along the isomorphism $T^\ast S^1 \times T^\ast S^1 \cong T^\ast (S^1 \times S^1)$, this becomes the standard
  \[ J^{(2)}(x) = \begin{bmatrix} 0 & I_2 \\ -I_2 & 0 \end{bmatrix} \]

  This gives a new parameter space; we now compose the smooth maps $w \colon T \real^{2} \times TS^{1} \to T\real^{2}$ and $E \colon T\real^{2} \times TS^{1} \to \real$ with themselves to get $w^{(2)} \colon T\real^2 \times (T^\ast S^1 \times T^\ast S^1) \to T\real^2$ defined by
  \[ T\real^2 \times T^\ast S^1 \times T^\ast S \to[w \times \id_{T^\ast S^1}] T\real^2 \times T^\ast S^1 \to[w] T\real^2 \]
  and $E^{(2)} \colon T\real^2 \times (T^\ast S^1 \times T^\ast S^1) \to \real$ defined by
  \[ T\real^2 \times T^\ast S^1 \times T^\ast S \to[\langle E, w \times \id_{T^\ast S^1} \rangle] \real \times T\real^2 \times T^\ast S^1 \to[\id_\real \times E] \real \times \real \to[+] \real \]

  In coordinates, this looks like the following. Assume that $x_0, v_0$ are the position and velocity of the first pivot, and $\theta_1, L_1, \theta_2, L_2$ are the natural coordinates for $T^\ast S^1 \times T^\ast S^1$. Then we define the following quantities, all of which depend on $(x_{0}, v_{0}, \theta_{1}, L_{1}, \theta_{2}, L_{2}) \in T\real^{2} \times T^{\ast}S^{1} \times T^{\ast}S^{1}$,

  \begin{eqalign*}
    x_1 &= x_0 + l(\cos \theta_1, \sin \theta_1) \\
    \omega_1 &= \frac{L_1}{I} \\
    v_1 &= v_0 + l\omega_1(-\sin\theta_1, \cos \theta_1) \\
    x_2 &= x_1 + l(\cos \theta_2, \sin \theta_2) \\
    \omega_2 &= \frac{L_2}{I} \\
    v_2 &= v_1 + l\omega_2(-\sin\theta_2, \cos \theta_2),
  \end{eqalign*}
  in order to define the output $w^{(2)}$ and energy $E^{(2)}$.
  \begin{eqalign*}
    w^{(2)} &= (x_2, v_2) \\
    E^{(2)} &= \frac{1}{2}m \abs{v_1}^2 + \frac{1}{2}m \abs{v_2}^2 + mgh_1 + mgh_2.
  \end{eqalign*}

  We finally apply the functor $\OpenPot \to \COrg$ in order to get an open dynamical system out of this. The state variables of this dynamical system are $(\theta_{1}, L_{1}, \theta_{2}, L_{2})$. It is a not-terribly-interesting exercise in calculus to actually compute all of the partials for $w^{(2)}$ and $E^{(2)}$; the interested reader could run the following Julia script if they wanted to see the actual expressions.

  \begin{verbatim}
using Symbolics

@variables x0 y0 vx0 vy0 θ1 L1 θ2 L2
@variables l I m g

x1 = x0 + l * cos(θ1)
y1 = y0 + l * sin(θ1)
ω1 = L1 / I
vx1 = vx0 + l * ω1 * (- sin(θ1))
vy1 = vy0 + l * ω1 * cos(θ1)

x2 = x1 + l * cos(θ2)
y2 = x1 + l * sin(θ2)
ω2 = L2 / I
vx2 = vx1 + l * ω2 * (- sin(θ2))
vy2 = vy1 + l * ω2 * cos(θ2)

E = 0.5 * m * (vx1^2 + vy1^2) + 0.5 * m * (vx2^2 + vy2^2) + m * g * y1 + m * g * y2

dEdx0 = expand_derivatives(Differential(x0)(E))
dEdy0 = expand_derivatives(Differential(y0)(E))

dEdvx0 = expand_derivatives(Differential(vx0)(E))
dEdvy0 = expand_derivatives(Differential(vy0)(E))

dEdθ1 = expand_derivatives(Differential(θ1)(E))
dEdL1 = expand_derivatives(Differential(L1)(E))

dEdθ2 = expand_derivatives(Differential(θ2)(E))
dEdL2 = expand_derivatives(Differential(L2)(E))

dθ1dt = dEdL1
dL1dt = -dEdθ1

dθ2dt = dEdL2
dL2dt = -dEdθ2
  \end{verbatim}

\end{example}

\section{Epilogue}\label{sec:epilogue}

In this work, we have shown the structure of open energy-driven system, involving a reaction and an energy functional, elegantly subsumes both Hamiltonian and gradient-based systems.
Out of these, we have built a symmetric monoidal category $\OpenPot$ which maps to $\OpenODE$, a symmetric monoidal category of open ODEs, which itself maps to $\COrg$, whose 1-cells are lenses `evolving smoothly' according to a given state space.

While this may seem like a lot of work/abstraction for little gain compared to ``just doing physics'' in the way one might in a traditional classical mechanics course, the advantage of working out this theory is to produce a ``plug and play'' physics system, where a library of components can be intuitively composed and all of the algebra is done by the computer.


\bibliographystyle{eptcsalpha}
\bibliography{act-submission}

\end{document}